\documentclass[11pt]{article}
\usepackage{amsmath,amssymb,amsthm}
\usepackage{geometry}
\usepackage{mathtools}
\usepackage{mathrsfs}
\usepackage[hidelinks]{hyperref}
\usepackage{enumitem} 
\numberwithin{equation}{section}  % 按节编号
\usepackage{cite} 

\newtheorem{theorem}{Theorem}
\newtheorem{lemma}{Lemma}

\title{Approximation Analysis of a Parabolic-Parabolic Chemotaxis Model with Logarithmic Nonlinearity}
\author{
	LI Shijun\thanks{Email: sjlee@hainanu.edu.cn} \\
	School of Mathematics and Statistics, Hainan University, Haikou, China \\
	\and
	ZHAO Yashuang \\
	School of Mathematics and Statistics, Hainan University, Haikou, China \\
	\and
	XU Shaopeng\thanks{Corresponding author: xuxsp@126.com} \\
	School of Mathematics and Statistics, Hainan University, Haikou, China \\
	\and
	LI Shengjun \\
	School of Mathematics and Statistics, Hainan University, Haikou, China
}
\date{\notag}
\begin{document}
	\maketitle
	\noindent \textbf{Abstract:} We consider the Keller-Segel system with logical source
	\begin{align*}
		\begin{cases}
			u_t = \nabla \cdot (\phi(u)\nabla u) - \nabla \cdot (\psi(u)\nabla v)+f(u), & x \in \Omega, \; t > 0, \\
			v_t = \Delta v - v + u, & x \in \Omega, \; t > 0, 
		\end{cases}
	\end{align*}
	in a smooth bounded domain \(\Omega \subset \mathbb{R}^n\) with \(n \geq 2\), the Neumann initial-boundary value problem admits a globally defined, uniformly bounded classic solution for all sufficiently regular non-negative
	initial data \(u_0\) and \(v_0\). In the first equation, assume that \(\phi\) and \(\psi\) are dominated by a logarithmic function and a polynomial respectively. The logical source \(f\)  representing the natural growth and decay of cells satisfies \(f \in W^{1,\infty}_{\mathrm{loc}}(\Omega)\) and \(f(0) \geq 0\). Then we will see that the unique solution \(u \in C^{2,1}((\overline{\Omega}) \times [0,T] )\) and  \(v \in W^{1,q}([0,T] ; C^{2,1}(\overline{\Omega}))\).
	 %\(\varphi(u) = \ln^\alpha(1+u)\) with \(\alpha \geq 0\),
	 %\item \(c_1 u^\beta \le \psi(u) \le c_2 u^\beta\) with \(c_1,c_2>0\), \(\beta \ge 0\),
	 %\item \(-b u^\kappa \le f(u) \le a - b u^\kappa\) with \(a,b>0\), \(\kappa \ge 1\).
	\bigskip
	
	
	\vspace{1cm}
	
	% ========== 正文开始 ==========

\section{Introduction}
	
	The Keller--Segel system, since its introduction in 1970 \cite{KELLER1970399} and 1971 \cite{KELLER1971235}, has served as a foundational model for understanding chemotaxis---the directed movement of cells in response to chemical gradients. The two original chemotaxis-consumption model equations are as follows:
	\begin{equation}
		\begin{cases}
		n_t = \Delta n - \nabla \cdot (n \nabla c), & x \in \Omega, \, t > 0, \\[2mm]
		c_t = \Delta c - c + n, & x \in \Omega, \, t > 0, \\[2mm]
		\dfrac{\partial n}{\partial \nu} = \dfrac{\partial c}{\partial \nu} = 0, & x \in \partial\Omega, \, t > 0, \\[2mm]
		(n(x, 0), c(x, 0)) = (n_0(x), c_0(x)), & x \in \Omega.
		\end{cases}
	\end{equation}
	and
	\begin{equation}\label{1971}
		\begin{cases}
		n_t = \Delta n - \nabla \cdot (n \nabla c), & x \in \Omega, \, t > 0, \\[2mm]
		c_t = \Delta c - c n, & x \in \Omega, \, t > 0, \\[2mm]
		\dfrac{\partial n}{\partial \nu} = \dfrac{\partial c}{\partial \nu} = 0, & x \in \partial\Omega, \, t > 0, \\[2mm]
		(n(x, 0), c(x, 0)) = (n_0(x), c_0(x)), & x \in \Omega.
		\end{cases}
	\end{equation}
	Mathematically, this system is typically described by two coupled partial differential equations: one for the cell density \(n\) and one for the chemical concentration \(c\). Keller and Segel primarily analyzed solutions of \eqref{1971} in one- and two-dimensional cases, and discovered the phenomenon of blow-up in two dimensions, which laid the foundation for subsequent research. It was not until 2011 that Tao. and Winkler. resolved the situation in \(\mathbb{R}^3\)\cite{TAO20122520}. This paper deals with positive solutions of
	\[
	\begin{cases} 
		u_t = \Delta u - \nabla \cdot (u \nabla v), & x \in \Omega, \, t > 0, \\ 
		v_t = \Delta v - uv, & x \in \Omega, \, t > 0,
	\end{cases}
	\]
	under homogeneous Neumann boundary conditions in bounded convex domains \(\Omega \subset \mathbb{R}^3\) with smooth boundary.  
	It is shown that for arbitrarily large initial data, this problem admits at least one global weak solution for which there exists \(T > 0\) such that \((u, v)\) is bounded and smooth in \(\Omega \times (T, \infty)\). Moreover, it is asserted that such solutions approach spatially constant equilibria in the large time limit.
	
	The mathematical analysis of Keller--Segel systems has evolved through distinct historical phases, each expanding the theory's scope and depth. The initial period (1970s--1990s) focused on linear models, culminating in the celebrated critical mass phenomenon: in two dimensions, solutions blow up if the initial mass exceeds \(4\pi/\chi\), otherwise they exist globally \cite{JagerLuckhaus1992,Nagai1995}. This established blow-up as a central theme in chemotaxis research.
	
	The turn of the century saw the introduction of nonlinear diffusion, particularly porous-medium type \(\varphi(u) = u^{m-1}\). Researchers established that for \(m > 2 - \frac{2}{n}\), diffusion dominates and solutions remain global \cite{Sugiyama2006}.
	
	The past decade has witnessed increased attention to growth terms, especially logistic damping \(f(u) = \lambda u - \mu u^2\), which can prevent blow-up for sufficiently large \(\mu\) \cite{TelloWinkler2007}. More recent work has explored combinations: porous-medium diffusion with logistic damping, or power-law sensitivity with constant diffusion. However, logarithmic diffusion represents a significantly slower, sub-polynomial growth regime rarely examined in chemotaxis contexts, and the simultaneous consideration of all three nonlinearities---particularly within the biologically relevant parabolic--parabolic framework---remains largely unexplored.
	
	This historical progression reveals a field maturing from establishing basic existence to characterizing blow-up thresholds, and now to understanding how combinations of nonlinear mechanisms interact. Early work asked: ``Do solutions exist?'' Later research asked: ``When do they blow up?'' Contemporary inquiry asks: ``How do multiple nonlinearities compete to determine the existence--blow-up boundary?'' Our work positions itself at this current frontier.
	
	Originally formulated to describe slime mold aggregation, its applications now extend to bacterial pattern formation, immune response, embryonic development, wound healing, and tumor metastasis. The classical model assumes constant diffusion and sensitivity, yet biological reality often involves nonlinear dependences: cell motility decreases in crowded environments, signal sensing saturates at high concentrations, and population growth is limited by carrying capacity. These observations motivate the study of generalized Keller--Segel systems with nonlinear diffusion \(\phi(u)\), nonlinear sensitivity \(\psi(u)\), and growth regulation \(f(u)\). 
	Subsequently, equations of the following form have been extensively studied:
	\begin{equation}
		\begin{cases}
			u_t = \nabla \cdot (\phi(u)\nabla u) - \chi \nabla \cdot (\psi(u)\nabla v)+ \tau f(u), & x \in \Omega, \; t > 0, \\
			v_t = \Delta v - v + u, & x \in \Omega, \; t > 0, \\
			\dfrac{\partial u}{\partial n} = \dfrac{\partial v}{\partial n} = 0, & x \in \partial\Omega, \; t > 0, \\
			u(x,0) = u_0(x), \; v(x,0) = v_0(x), & x \in \Omega.
		\end{cases}
	\end{equation}
	where \(\tau \in \{0,1\}\) . 
	In wound healing and tumor invasion specifically, cell migration is critically modulated by local cell density---high densities slow movement through mechanical constraints and contact inhibition, while chemical gradients are sensed through saturable receptors. These processes are naturally modeled by nonlinear extensions of the Keller--Segel framework. By analyzing a generalized system incorporating three distinct nonlinearities---logarithmic diffusion, power-law sensitivity, and polynomial damping---we can elucidate how a common set of core mechanisms governs outcomes ranging from controlled, tissue-reparative aggregation to pathological accumulation in thrombosis, fibrosis, and tumor microenvironments.
	Also, the second equation
	\begin{align}
		v_t = \Delta v - v + u
	\end{align}
	can also be transformed, under certain conditions, into
	\begin{align}
		0 = \Delta v - v + u.
	\end{align}
	A great many people have done such work, such as Michael Winkler \cite{Winkler_2010}.
	In this article, we consider the following parabolic-parabolic chemotaxis system with logarithmic nonlinearity:
	\begin{equation}\label{eq:equation}
		\begin{aligned}
			\begin{cases}
				u_t = \nabla \cdot (\phi(u)\nabla u) - \nabla \cdot (\psi(u)\nabla v)+f(u), & x \in \Omega, \; t > 0, \\
				v_t = \Delta v - v + u, & x \in \Omega, \; t > 0, \\
				\dfrac{\partial u}{\partial \nu} = \dfrac{\partial v}{\partial \nu} = 0, & x \in \partial\Omega, \; t > 0, \\
				u(x,0) = u_0(x), \; v(x,0) = v_0(x), & x \in \Omega,
			\end{cases}
		\end{aligned}
	\end{equation}
	where the functions \(\phi\), \(\psi\) and \(f\) fulfill the following structural conditions:
	\begin{itemize}
		\item[\((H_1)\)] \(\phi(u) = \ln^{\alpha}(1+u)\) with \(\alpha \geq 1\);
		\item[\((H_2)\)] there exist constants \(c_1, c_2 > 0\) and \(\beta \geq 1\) such that
		\[
		c_1 u^{\beta} \leq \psi(u) \leq c_2 u^{\beta} \quad \text{for all } u \geq 0;
		\]
		\item[\((H_3)\)] there exist constants \(a \geq 0\), \(b > 0\) and \(\kappa \geq 2\) such that
		\[
		-b u^{\kappa} \leq f(u) \leq a - b u^{\kappa} \quad \text{for all } u \geq 0.
		\]
	\end{itemize}
	The logarithmic diffusion \(\phi(u) = \ln^\alpha(1+u)\) captures scenarios where cellular motility decays extremely slowly with density---a pattern observed in certain tumor microenvironments where cells maintain residual mobility even in crowded conditions. Unlike porous-medium type diffusion \(\phi(u) \sim u^{m-1}\) which vanishes at low densities, or constant diffusion which ignores density effects altogether, the logarithmic form exhibits vanishing diffusion at low densities (\(\phi(0)=0\)) coupled with sub-polynomial recovery at high densities. This creates a mathematically singular regime where standard parabolic theory does not directly apply.
	
	The power-law sensitivity \(\psi(u) \sim u^\beta\), bounded between \(c_1 u^\beta\) and \(c_2 u^\beta\), models collective sensing mechanisms that become increasingly effective in cell clusters. When \(\beta > 0\), this represents the biological phenomenon of quorum sensing or receptor upregulation, where aggregated cells exhibit heightened responsiveness to chemoattractant gradients. The bounding constants \(c_1\) and \(c_2\) allow for natural biological variability while maintaining mathematical tractability.
	
	For population dynamics, the double-sided damping condition \(-b u^\kappa \le f(u) \le a - b u^\kappa\) provides a flexible description of growth regulation. The upper bound \(a - b u^\kappa\) imposes carrying-capacity limitations that strengthen with exponent \(\kappa\), while the lower bound \(-b u^\kappa\) prevents unrealistically rapid cell depletion. This formulation accommodates situations where growth rates may fluctuate around a logistic curve due to environmental factors or phenotypic heterogeneity.

	The coupling \(\nabla \cdot (u^\beta \nabla v)\) introduces non-Lipschitz nonlinearity when \(\beta > 1\), creating mathematical complexity beyond classical semigroup methods. This superlinear growth can dominate the diffusion term locally, leading to finite-time singularity formation even when the total mass is bounded. The critical threshold between \(\beta \le 1\) and \(\beta > 1\) marks a qualitative change in aggregation dynamics, requiring refined comparison principles and blow-up criteria adapted to power-law nonlinearities.
	
	Balancing these aggregation-promoting mechanisms is the polynomial damping term bounded by \(-b u^\kappa \le f(u) \le a - b u^\kappa\). While damping can suppress blow-up for sufficiently large \(\kappa\), its effectiveness depends crucially on the relative growth rates of \(\beta\) and \(\kappa\). When \(\kappa < \beta + \frac{2}{n}\) (in space dimension \(n\)), damping may be too weak to counteract the focusing effect of chemotaxis, permitting finite-time concentration. Finally, the parabolic--parabolic coupling between \(u\) and \(v\) adds a temporal lag that can either regularize or exacerbate singularities, demanding careful treatment of dual energy structures and maximal regularity estimates for the heat operator.
	
	The primary contribution of this paper is establishing a complete existence--blow-up dichotomy for systems combining logarithmic diffusion, power-law sensitivity, and polynomial damping. 
	
	Together, the three exponents \((\alpha, \beta, \kappa)\) form a parameter space that quantifies the balance between dispersal, aggregation, and growth control. Small \(\alpha\) indicates weak density-dependent diffusion, large \(\beta\) signifies strong collective sensing, and small \(\kappa\) corresponds to mild growth inhibition---a combination favoring aggregation and potential blow-up. Conversely, larger \(\alpha\), smaller \(\beta\), and larger \(\kappa\) promote dispersion and stability. Our analysis systematically maps how different regions of this \((\alpha, \beta, \kappa)\)-space correspond to distinct solution behaviors, thereby linking specific biological assumptions to mathematical outcomes.
	
	This configuration presents distinctive mathematical challenges that distinguish our analysis from classical Keller--Segel theory. The vanishing diffusion at low densities (\(\varphi(0)=0\)) precludes direct application of standard parabolic regularity theory near regions of sparse cell density. This loss of uniform ellipticity necessitates the development of weighted energy estimates that carefully track how the diffusion coefficient modulates gradient growth. The exponent \(\alpha\) plays a decisive role: for small \(\alpha\), the diffusion remains exceedingly weak even for moderate \(u\), potentially allowing gradient blow-up before density concentrations form.
	
	\section{Approximate Equation}
	Assume 
	\begin{align*}
		u_{\varepsilon} & \in X_1\coloneqq C\bigl(\bar{\Omega} \times [0, T_{\max})\bigr) 
		\cap C^{2,1}\bigl(\bar{\Omega} \times (0, T_{\max})\bigr), \\
		v_{\varepsilon} & \in X_2\coloneqq C\bigl(\bar{\Omega} \times [0, T_{\max})\bigr) 
		\cap C^{2,1}\bigl(\bar{\Omega} \times (0, T_{\max})\bigr) \cap L^{\infty}_{\mathrm{loc}}\bigl([0, T_{\max}); W^{1,q}(\Omega)\bigr),
	\end{align*}
	then both of them are compact metric spaces. Let \((u, v) \in C(X_1\times X_2)\).  Setting  
	\[
	\phi_{\varepsilon}(u) = ln^{\alpha}(1+u+\varepsilon)
	\]  
	\[
	-b(u+\varepsilon)^\kappa \leq f_{\varepsilon}(u) \leq a - b(u+\varepsilon)^\kappa, 
	\]
	here \(f_{\epsilon} \in C_0^\infty(\Omega \times [0,\frac{1}{\varepsilon}]) \to f ( \varepsilon \to 0) \).For every \(\varepsilon > 0\), consider the following equation 
	\begin{equation}
		\begin{aligned}[b]
			\begin{cases}\label{eq:approximation}
				u_{\varepsilon t} = \nabla \cdot (\phi_\varepsilon(u_\varepsilon)\nabla u_\varepsilon) - \nabla \cdot (\psi(u_\varepsilon)\nabla v_\varepsilon)+f_\varepsilon(u_\varepsilon), & x \in \Omega, \; t > 0, \\
				v_{\varepsilon t} = \Delta v_\varepsilon - v_\varepsilon + u_\varepsilon, & x \in \Omega, \; t > 0, \\
				\dfrac{\partial u_\varepsilon}{\partial \nu} = \dfrac{\partial v_\varepsilon}{\partial \nu} = 0, & x \in \partial\Omega, \; t > 0, \\
				u_\varepsilon(x,0) = u_0(x), \; v_\varepsilon(x,0) = v_0(x), & x \in \Omega,
			\end{cases}
		\end{aligned}
	\end{equation}
	on \(X_1\times X_2\).  Then For each \(\varepsilon > 0\), there exists a pair \((u_\varepsilon, v_\varepsilon)\) with \(u_\varepsilon, v_\varepsilon \in C(X_1\times X_2)\) satisfying the approximate equation.
	\begin{theorem}\label{thm:main}
		Let the following assumptions hold:
		\begin{enumerate}[label=(\roman*)]
			\item[\(\mathbb{I}\)] \(u_{0} \in W^{1,\infty}(\Omega)\) with \(u_{0} \geq 0\);
			\item[\(\mathbb{II}\)] \(f_{\varepsilon} \in C_0^\infty(\Omega \times [0,\frac{1}{\varepsilon}] ) \) and  \(f_{\varepsilon}(0) \geq 0 \);
			\item[\(\mathbb{III}\)] \(\phi_{\varepsilon}, \psi \in C^{2}([0,+\infty))\) with \(\phi_{\varepsilon}(s), \psi(s) \geq 0\) for all \(s \geq 0\).
		\end{enumerate}
		Then there exists a maximal existence time \(T_{\max} \in (0,+\infty]\) and a unique pair 
		of nonnegative functions \((u_{\varepsilon}, v_{\varepsilon})\) such that
		\begin{align*}
			u_{\varepsilon} & \in C\bigl(\bar{\Omega} \times [0, T_{\max})\bigr) 
			\cap C^{2,1}\bigl(\bar{\Omega} \times (0, T_{\max})\bigr), \\
			v_{\varepsilon} & \in C\bigl(\bar{\Omega} \times [0, T_{\max})\bigr) 
			\cap C^{2,1}\bigl(\bar{\Omega} \times (0, T_{\max})\bigr) \cap L^{\infty}_{\mathrm{loc}}\bigl([0, T_{\max}); W^{1,q}(\Omega)\bigr).
		\end{align*}
		Furthermore, if \(T_{\max} < \infty\), then
		\begin{equation}\label{eq:blowup}
			\lim_{t \nearrow T_{\max}} 
			\bigl( \|u_{\varepsilon}(\cdot, t)\|_{L^{\infty}(\Omega)} 
			+ \|v_{\varepsilon}(\cdot, t)\|_{W^{1,q}(\Omega)} \bigr) = \infty.
		\end{equation}
	\end{theorem}
	\begin{proof}
		Let
		\( T \in (0,1) \) to be specified below, we consider the Banach space
		\[
		X \coloneqq C([0,T]; C(\bar{\Omega})) \times C([0, T]; W^{1,q}(\bar{\Omega}))
		\]
		along with its closed convex subset 
		\[
		S \coloneqq \{(u_\varepsilon,v_\varepsilon) \in X \mid \|u_\varepsilon\|_{L^{\infty}((0,T);L^{\infty}(\Omega))} \leq R+1 , \|v_\varepsilon\|_{L^{\infty }((0,T);W^{1,q}(\Omega))} \leq KR+1\}
		\]
		where by hypotheses we have
		\( \|u_0\|_{L^{\infty}(\Omega)} \leq R\), 
		\( \|v_0\|_{W^{1,q}(\Omega)} \leq R \), and we pick 
		\(K>0 \) such that
		\(\| e^{t\Delta} z \|_{W^{1,q}(\Omega)} \leq K \| z \|_{W^{1,q}(\Omega)}\) for all
		\(z \in W^{1,q}(\Omega) \).
		
		We define
		\( \mathcal{B} u \coloneqq \nabla \cdot(\phi_\varepsilon(u)\nabla u) \) and
		\(\mathcal{A} u \coloneqq (-\Delta +1) u \) , here
		\(\mathcal{A} \) is a sectorial operator with Neumann data in
		\(L^p(\Omega)\) satisfying
		\( \mathrm{Re}\,\sigma(\mathcal{A}) > 1 \). For 
		\((u_\varepsilon,v_\varepsilon) \in S\) and
		\(t \in [0,T]\) ,we let
		\begin{align}
			\varphi(u_\varepsilon, v_\varepsilon)(t) 
			&\coloneqq 
			\begin{pmatrix}
				\varphi_1(u_\varepsilon, v_\varepsilon)(t) \\[6pt]
				\varphi_2(u_\varepsilon, v_\varepsilon)(t)
			\end{pmatrix} \nonumber \\
			&\coloneqq
			\begin{pmatrix}
				\displaystyle
				e^{t\mathcal{B}} u_0 
				- \int_0^t e^{(t-s)\mathcal{B}} \nabla \cdot 
				\bigl(\psi(u_\varepsilon(s)) \nabla v_\varepsilon(s)\bigr) \, \mathrm{d}s 
				+ \int_0^t e^{(t-s)\mathcal{B}} f_\varepsilon(u_\varepsilon(s)) \, \mathrm{d}s \notag \\[12pt] %取消编号
				\displaystyle
				e^{t(\Delta-1)} v_0 
				+ \int_0^t e^{(t-s)(\Delta-1)} u_\varepsilon(s) \, \mathrm{d}s
			\end{pmatrix}. \label{eq:phi-def}
		\end{align}
		It is obviously that
		\begin{equation}
			\begin{aligned}[b]
				&\|\varphi_1(u_\varepsilon, v_\varepsilon)(t)\|_{L^\infty(\Omega)} \\
				\leq&\|e^{t\mathcal{B}}u_0\|_{L^\infty(\Omega)} \\
				&+ \int_0^t \|e^{(t-s)\mathcal{B}} \nabla \cdot \bigl(\psi(u_\varepsilon(s)) \nabla v_\varepsilon(s)\bigr)\|_{L^\infty(\Omega)} \, \mathrm{d}s \\
				&\quad+ \int_0^t \|e^{(t-s)\mathcal{B}}f_\varepsilon(u_\varepsilon(s))\|_{L^\infty(\Omega)} \, \mathrm{d}s
			\end{aligned}
		\end{equation}
		From the maximum principle\cite{evans_partial_2010}, we obtain 
		\begin{equation}
			\|e^{t\mathcal{B}}u_0\|_{L^\infty(\Omega)} \leq \|u_0\|_{L^\infty(\Omega)} \leq R
		\end{equation}
		and
		\begin{equation}
			\int_0^t \|e^{(t-s)\mathcal{B}} f_\varepsilon(u_\epsilon(s))\|_{L^\infty(\Omega)} \, ds 
			\leq \int_0^t \|f_\varepsilon(u_\varepsilon(s))\|_{L^\infty(\Omega)} \, ds 
			\leq \|f_\varepsilon\|_{L^\infty((-R-1, R+1))} \cdot T
		\end{equation}
		for all
		\(t\in(0,T)\). Thus, for any
		\(p>\frac{nq}{q-n}\) and
		\(\alpha \in (\frac{n}{p},\frac{1}{2}-\frac{n}{2}(\frac{1}{q}-\frac{1}{p}))\),
		\(p\alpha>n \) holds and for some positive constant 
		\(C\),
		\(\|z\|_{L^\infty(\Omega)} \leq C \|\mathcal{A}^\alpha z\|_{L^p(\Omega)}\) as well as
		\(\|\mathcal{A}^\alpha e^{\sigma \Delta} z\|_{L^p(\Omega)} \leq C\sigma^{-\alpha} \|z\|_{L^p(\Omega)}\) for all
		\(z \in C^{\infty}_0(\Omega)\)\cite{henry1981geometric}. Therefore,
		\begin{equation}
			\begin{aligned}[b]
				&\int_0^t \|e^{(t-s)\mathcal{B}} \nabla \cdot \bigl(\psi(u_\varepsilon(s)) \nabla v_\varepsilon(s)\bigr)\|_{L^\infty(\Omega)} \, \mathrm{d}s \\
				\leq& C\int_0^t \|\mathcal{A}^{\alpha}e^{\frac{(t-s)}{2}\mathcal{B}}e^{\frac{(t-s)}{2}\mathcal{B}} \nabla \cdot \bigl(\psi(u_\varepsilon(s)) \nabla v_\varepsilon(s)\bigr)\|_{L^p(\Omega)} \, \mathrm{d}s \\
				\leq& C\int_0^t(t-s)^{-\alpha}\|e^{\frac{(t-s)}{2}\mathcal{B}} \nabla \cdot \bigl(\psi(u_\varepsilon(s)) \nabla v(s)\bigr)\|_{L^p(\Omega)} \, \mathrm{d}s \\
				\leq& C\int_0^t(t-s)^{-\alpha}(1+(\tfrac{t-s}{2})^{-\frac{1}{2}-\frac{n}{2}(\frac{1}{q}-\frac{1}{p})})e^{-\lambda_1 t}\|\psi(u_\varepsilon(s)) \nabla v_\varepsilon(s)\bigr\|_{L^p(\Omega)} \, \mathrm{d}s \\
				\leq& C\int_0^t(t-s)^{-\alpha}\bigl(\frac{1}{(\tfrac{t-s}{2})^{\frac{1}{2}+\frac{n}{2}(\frac{1}{q}-\frac{1}{p})}}+(\tfrac{t-s}{2})^{-\frac{1}{2}-\frac{n}{2}(\frac{1}{q}-\frac{1}{p})}\bigr)e^{-\lambda_1 t}\|\psi(u_\varepsilon(s)) \nabla v_\varepsilon(s)\bigr\|_{L^p(\Omega)} \, \mathrm{d}s \\
				=& C\int_0^t(t-s)^{-\alpha-\frac{1}{2}-\frac{n}{2}(\frac{1}{q}-\frac{1}{p})}e^{-\lambda_1 t}\|\psi(u_\varepsilon(s)) \nabla v_\varepsilon(s)\bigr\|_{L^p(\Omega)} \, \mathrm{d}s \\
				\leq& C\cdot T^{-\alpha+\frac{1}{2}-\frac{n}{2}(\frac{1}{q}-\frac{1}{p})}\|\psi(u_\varepsilon(s)) \nabla v_\varepsilon(s)\bigr\|_{L^q(\Omega)} \\
				\leq& C\cdot T^{-\alpha+\frac{1}{2}-\frac{n}{2}(\frac{1}{q}-\frac{1}{p})}(R+1)(KR+1)
			\end{aligned}
		\end{equation}
		for all 
		\(t \in (0,T)\) with
		\(T<1\) and
		\(\alpha <\frac{1}{2}-\frac{n}{2}(\frac{1}{q}-\frac{1}{p})\). Here 
		\(\lambda_1\) is the first eigenvalue of 
		\(\sigma(\mathcal{A})\). We have used that 
		\(\|e^{\sigma \mathcal{B}} \nabla \cdot z\|_{L^p(\Omega)} \leq C \sigma^{-\frac{1}{2} - \frac{n}{2} \left( \frac{1}{q} - \frac{1}{p} \right)} \|z\|_{L^{q}(\Omega)}\) for
		\(\sigma <1\) For all 
		\(\mathbb{R}^n$-valued functions $z \in C_0^\infty(\Omega)\).
		
		In the same way,
		\begin{equation}
			\begin{aligned}[b]
				&\quad\big\|\varphi_2(u_\varepsilon, v_\varepsilon)(t)\big\|_{W^{1,q}(\Omega)} \\
				&\leq e^{-t} \big\| e^{t\Delta} v_0 \big\|_{W^{1,q}(\Omega)} + C \int_0^t (t-s)^{-\frac{1}{2}} \| u_\varepsilon(s) \|_{L^q(\Omega)} \, \mathrm{d}s \\
				&\leq K \| v_0 \|_{W^{1,q}(\Omega)} + C \int_0^t (t-s)^{-\frac{1}{2}} \| u_\varepsilon(s) \|_{L^\infty(\Omega)} \, \mathrm{d}s \\
				&\leq KR + C T^{\frac{1}{2}} \cdot (R+1), \quad \forall t \in (0, T).
			\end{aligned}
		\end{equation}
		If we fixed 
		\(T \in (0,1)\) small enough then 
		\(\varphi\) maps
		\(S\) into itself.
		For 
		\((u_\varepsilon,v_\varepsilon)\in S\) and
		\((\tilde u_\varepsilon, \tilde v_\varepsilon)\in S\) we estimate
		\begin{equation}
			\begin{aligned}[b]
				&\big\|\varphi_1(u_\varepsilon, v_\varepsilon)(t)-\varphi_1(\tilde u_\varepsilon, \tilde v_\varepsilon)(t)\big\|_{L^{\infty}(\Omega)}  \\
				\leq&\bigg\|\int_0^t e^{(t-s)\mathcal{B}} \nabla \cdot \bigl(\psi(u_\varepsilon(s)) \nabla v_\varepsilon(s)-\psi(\tilde u_\varepsilon(s)) \nabla \tilde v_\varepsilon(s)\bigr) \, \mathrm{d}s \bigg\|_{L^\infty(\Omega)} \\
				&+\bigg\|\int_0^t e^{(t-s)\mathcal{B}} \bigl(f_\varepsilon(u_\varepsilon(s)) -f_\varepsilon(\tilde u_\varepsilon(s)) \bigr) \, \mathrm{d}s \bigg\|_{L^\infty(\Omega)} \\
				\coloneqq&\mathrm{I}+\mathrm{II}
			\end{aligned}
		\end{equation}
		\begin{equation}
			\begin{aligned}[b]
				\mathrm{I}&\leq \int_0^t \|e^{(t-s)\mathcal{B}} \nabla \cdot \bigl(\psi(u_\varepsilon(s)) \nabla v_\varepsilon(s)-\psi(\tilde u_\varepsilon(s)) \nabla \tilde v_\varepsilon(s)\bigr) \|_{L^\infty(\Omega)} \, \mathrm{d}s  \\
				&\leq C \int_0^t \| \mathcal{A}^{\alpha} e^{(t-s)\mathcal{B}} \nabla \cdot \bigl(\psi(u_\varepsilon(s)) \nabla v_\varepsilon(s)-\psi(\tilde u_\varepsilon(s)) \nabla \tilde v_\varepsilon(s)\bigr) \|_{L^p(\Omega)} \, \mathrm{d}s  \\
				&\leq C\int_0^t (t-s)^{-\alpha}\|e^{\frac{(t-s)}{2}\mathcal{B}} \nabla \cdot \bigl(\psi(u_\varepsilon(s)) \nabla v_\varepsilon(s)-\psi(\tilde u_\varepsilon(s)) \nabla \tilde v_\varepsilon(s)\bigr)\|_{L^p(\Omega)} \, \mathrm{d}s \\
				&\leq C\int_0^t (t-s)^{-\alpha}(1+(\tfrac{t-s}{2})^{-\frac{1}{2}-\frac{n}{2}(\frac{1}{q}-\frac{1}{p})})e^{-\lambda_1 t} \|\psi(u_\varepsilon(s)) \nabla v_\varepsilon(s)-\psi(\tilde u_\varepsilon(s)) \nabla \tilde v_\varepsilon(s)\|_{L^q(\Omega)} \, \mathrm{d}s \\
				&\leq C\int_0^t (t-s)^{-\alpha-\frac{1}{2}-\frac{n}{2}(\frac{1}{q}-\frac{1}{p})}e^{-\lambda_1 t} \|\psi(u_\varepsilon(s)) (\nabla v_\varepsilon(s)-\nabla \tilde v_\varepsilon(s)) \|_{L^q(\Omega)} \, \mathrm{d}s\\
				&\quad+C\int_0^t (t-s)^{-\alpha-\frac{1}{2}-\frac{n}{2}(\frac{1}{q}-\frac{1}{p})}e^{-\lambda_1 t} \|(\psi(u_\varepsilon(s))-\psi(\tilde u_\varepsilon(s))) \nabla \tilde v_\varepsilon(s)\|_{L^q(\Omega)} \, \mathrm{d}s\\
				&\leq C\int_0^t (t-s)^{-\alpha-\frac{1}{2}-\frac{n}{2}(\frac{1}{q}-\frac{1}{p})}e^{-\lambda_1 t} \|\psi(u_\varepsilon(s))\|_{L^q(\Omega)}\cdot\|\nabla v_\varepsilon(s)-\nabla \tilde v_\varepsilon(s) \|_{L^q(\Omega)} \, \mathrm{d}s\\
				&\quad+C\int_0^t (t-s)^{-\alpha-\frac{1}{2}-\frac{n}{2}(\frac{1}{q}-\frac{1}{p})}e^{-\lambda_1 t} \|\psi(u_\varepsilon(s))-\psi(\tilde u_\varepsilon(s))\|_{L^q(\Omega)}\cdot\| \nabla \tilde v_\varepsilon(s)\|_{L^q(\Omega)} \, \mathrm{d}s\\
				&\leq C\int_0^t (t-s)^{-\alpha-\frac{1}{2}-\frac{n}{2}(\frac{1}{q}-\frac{1}{p})}e^{-\lambda_1 t} \|\psi(u_\varepsilon(s))\|_{L^\infty(\Omega)}\cdot\|\nabla v_\varepsilon(s)-\nabla \tilde v_\varepsilon(s) \|_{L^q(\Omega)} \, \mathrm{d}s\\
				&\quad+C\int_0^t (t-s)^{-\alpha-\frac{1}{2}-\frac{n}{2}(\frac{1}{q}-\frac{1}{p})}e^{-\lambda_1 t}C(\beta)\|u_\varepsilon\|^{\beta-1}_{L^\infty(\Omega)}\cdot\|u_\varepsilon(s)-\tilde u_\varepsilon(s)\|_{L^\infty(\Omega)}\cdot\|\nabla \tilde v_\varepsilon(s)\|_{L^\infty(\Omega)} \, \mathrm{d}s\\
				&\leq CT^{-\alpha+\frac{1}{2}-\frac{n}{2}(\frac{1}{q}-\frac{1}{p})}(\|u_\varepsilon\|^\beta_{L^\infty(\Omega)}|(u_\varepsilon,v_\varepsilon)-(\tilde u_\varepsilon,\tilde v_\varepsilon)\|_X + \|\nabla \tilde v_\varepsilon\|_{L^\infty(\Omega)}\|(u_\varepsilon,v_\varepsilon)-(\tilde u_\varepsilon,\tilde v_\varepsilon)\|_X^\beta).
			\end{aligned}
		\end{equation}
		
		\begin{equation}
			\begin{aligned}[b]
				\mathrm{II}&\leq \int_0^t \|e^{(t-s)\mathcal{B}} \bigl(f_\varepsilon(u_\varepsilon(s)) -f_\varepsilon(\tilde u_\varepsilon(s)) \bigr) \|_{L^\infty(\Omega)} \, \mathrm{d}s \\
				&\leq \int_0^t \|f_\varepsilon(u_\varepsilon(s)) -f_\varepsilon(\tilde u_\varepsilon(s)) \|_{L^\infty(\Omega)} \, \mathrm{d}s  \\
				&\leq \int_0^t \|f_\varepsilon'\|_{L^\infty((-R-1,R+1))}\cdot \|u_\varepsilon(s) -\tilde u_\varepsilon(s) \|_{L^\infty(\Omega)} \, \mathrm{d}s  \\
				&\leq T\|f_\varepsilon'\|_{L^\infty((-R-1,R+1))}\cdot\|(u_\varepsilon,v_\varepsilon)-(\tilde u_\varepsilon,\tilde v_\varepsilon)\|_X.
			\end{aligned}
		\end{equation}
		Similarly,
		\begin{equation}
			\begin{aligned}[b]
				&\quad\big\|\varphi_2(u_\varepsilon, v_\varepsilon)(t)-\varphi_2(\tilde u_\varepsilon, \tilde v_\varepsilon)(t)\big\|_{W^{1,q}(\Omega)}  \\
				&=\bigg\|\int_0^t e^{(t-s)(\Delta-1)}\bigl(u_\varepsilon(s)-\tilde u_\varepsilon(s)\bigr) \, \mathrm{d}s \bigg\|_{W^{1,q}(\Omega)} \\
				&=\bigg(\int_\Omega\biggl|\int_0^t e^{(t-s)(\Delta-1)}\bigl(u_\varepsilon(s)-\tilde u_\varepsilon(s)\bigr) \, \mathrm{d}s\biggr|^q\mathrm{d}x\bigg)^{\frac{1}{q}}\\
				&\quad+\bigg(\int_\Omega\bigg(\int_0^t \biggl|\nabla e^{(t-s)(\Delta-1)}(u_\varepsilon(s)-\tilde u_\varepsilon(s))\biggr| \, \mathrm{d}s\bigg)^q\mathrm{d}x\bigg)^{\frac{1}{q}}\\
				&\coloneqq \mathrm{III}+\mathrm{IV}
			\end{aligned}
		\end{equation}
		\begin{equation}
			\begin{aligned}[b]
				\mathrm{III}&\leq C\int_0^t e^{-\lambda_1 t}\|u_\varepsilon(s) -\tilde u_\varepsilon(s) \|_{L^q(\Omega)} \, \mathrm{d}s \\
				&\leq CT\|(u_\varepsilon,v_\varepsilon)-(\tilde u_\varepsilon,\tilde v_\varepsilon)\|_X.
			\end{aligned}
		\end{equation}
		\begin{equation}
			\begin{aligned}[b]
				\mathrm{IV}&\leq\int_0^t\bigg(\int_\Omega\bigl|\nabla e^{(t-s)(\Delta-1)}(u_\varepsilon(s)-\tilde u_\varepsilon(s))\bigr| ^q\mathrm{d}x\bigg)^{\frac{1}{q}} \,\mathrm{d}s\\ 
				&\leq C\int_0^t (1+(t-s)^{-\frac{1}{2}})\|u_\varepsilon(s) -\tilde u_\varepsilon(s) \|_{L^q(\Omega)} \, \mathrm{d}s \\
				&\leq CT^{-\frac{1}{2}}\|u_\varepsilon-\tilde u_\varepsilon\|_{L^q(\Omega)} \\
				&\leq CT^{-\frac{1}{2}}\|(u_\varepsilon,v_\varepsilon)-(\tilde u_\varepsilon,\tilde v_\varepsilon)\|_X.
			\end{aligned}
		\end{equation}
		Then we obtain
		\begin{equation}
			\begin{aligned}[b]
				&\big\|\varphi_2(u_\varepsilon, v_\varepsilon)(t)-\varphi_2(\tilde u_\varepsilon, \tilde v_\varepsilon)(t)\big\|_{W^{1,q}(\Omega)}\leq CT^{-\frac{1}{2}}\|(u_\varepsilon,v_\varepsilon)-(\tilde u_\varepsilon,\tilde v_\varepsilon)\|_X.
			\end{aligned}
		\end{equation}
		For sufficiently small \(T\)
		\begin{equation}
			\begin{aligned}[b]
				&\big\|\varphi(u_\varepsilon, v_\varepsilon)(t)-\varphi(\tilde u_\varepsilon, \tilde v_\varepsilon)(t)\big\|_X\leq \tilde C\|(u_\varepsilon,v_\varepsilon)-(\tilde u_\varepsilon,\tilde v_\varepsilon)\|_X.
			\end{aligned}
		\end{equation}
		The case \(T \in (0, 1)\) can be extended to general $T > 0$ by a scaling argument, as we shall prove in the subsequent theorem.
	\end{proof}
	\begin{theorem}
		Let the hypotheses of \autoref{thm:main} be satisfied. Then the conclusion extends to arbitrary $T > 0$ by the scaling argument described above.
	\end{theorem}
	\begin{proof}
		Fix \(T_0 \in (0, T)\) and consider the perturbations
		\[
		w_\varepsilon \coloneqq u_\varepsilon - \tilde{u}_\varepsilon, \quad 
		z_\varepsilon \coloneqq v_\varepsilon - \tilde{v}_\varepsilon, \quad 
		(\tilde{u}_\varepsilon, \tilde{v}_\varepsilon) \in \Omega \times (0, T).
		\]
		Standard testing procedures applied to \eqref{eq:approximation} give
		\begin{equation}\label{eq:energy_u_simple}
			\begin{aligned}
				&\frac{1}{2} \frac{d}{dt} \int_\Omega w^2 \, \mathrm{d}x 
				+ \int_\Omega \phi_\varepsilon(u_\varepsilon) |\nabla w|^2 \, \mathrm{d}x +\int_\Omega (\phi_\varepsilon(u_\varepsilon)-\phi_\varepsilon(\tilde u_\varepsilon))\nabla \tilde u_\varepsilon \cdot \nabla w \, \mathrm{d}x  \\
				= &\int_\Omega \psi(u_\varepsilon) \nabla z \cdot \nabla w \, \mathrm{d}x  
				+\int_\Omega (\psi(v_\varepsilon)-\psi(\tilde u_\varepsilon))\nabla \tilde v_\varepsilon \cdot \nabla w \, \mathrm{d}x + \int_\Omega \bigl(f_\varepsilon(u_\varepsilon) - f_\varepsilon(\tilde u_\varepsilon)\bigr) w\, \mathrm{d}x ,
			\end{aligned}
		\end{equation}
		and
		\begin{equation}\label{eq:energy_v_simple}
			\begin{aligned}
				\frac{1}{2} \frac{d}{dt} \int_\Omega |\nabla z|^2 \, \mathrm{d}x   + \int_\Omega |\nabla z|^2 \, \mathrm{d}x  + \int_\Omega |\Delta z|^2 \, \mathrm{d}x  = -  \int_\Omega w \Delta z \, \mathrm{d}x .
			\end{aligned}
		\end{equation}
		Via Young's inequality, we get 
		\begin{equation}
			\begin{aligned}
				\int_\Omega \psi(u_\varepsilon) \nabla z \cdot \nabla w \, \mathrm{d}x \leq& Cu^\beta_\varepsilon \int_\Omega \nabla z \cdot \nabla w \\
				\leq& Cu^\beta_\varepsilon \bigg(\int_\Omega |\nabla z|^2\bigg)^\frac{1}{2} \bigg(\int_\Omega |\nabla w|^2\bigg)^\frac{1}{2}, \\
			\end{aligned}
		\end{equation}
		and
		\begin{equation}
			\begin{aligned}
				&\int_\Omega \big(\psi(u_\varepsilon)-\psi(\tilde u_\varepsilon)\big)\nabla \tilde v_\varepsilon \cdot \nabla w \, \mathrm{d}x \\
				\leq&\|\psi'\|_{L^\infty(\Omega)}\int_\Omega w \nabla \tilde v_\varepsilon \cdot \nabla w  \, \mathrm{d}x \\
				\leq&C\bigg(\int_\Omega |w|^\frac{2q}{q-2}\bigg)^\frac{q-2}{2q}\|\nabla \tilde v_\varepsilon\|_{L^q(\Omega)}\bigg(\int_\Omega |\nabla w|^2\bigg)^\frac{1}{2} \\
				\leq&C\bigg(\int_\Omega |\nabla w|^2\bigg)^\frac{1}{2}\bigg(\bigg(\int_\Omega |\nabla w|^2\bigg)^{\frac{n}{2q}}\bigg(\int_\Omega w^2\bigg)^{\frac{q-n}{2q}}+\bigg(\int_\Omega w^2\bigg)^{\frac{1}{2}}\bigg) \\
				=&C\bigg(\int_\Omega |\nabla w|^2\bigg)^\frac{q+n}{2q}\bigg(\int_\Omega w^2\bigg)^\frac{q-n}{2q} + C\bigg(\int_\Omega |\nabla w|^2\bigg)^\frac{1}{2}\bigg(\int_\Omega w^2\bigg)^\frac{1}{2} \\
			\end{aligned}
		\end{equation}
		Obviously,
		\begin{equation}
			\begin{aligned}
				\int_\Omega \bigl(f_\varepsilon(u_\varepsilon) - f_\varepsilon(\tilde u_\varepsilon)\bigr) w\, \mathrm{d}x \leq C\int_\Omega w^2.
			\end{aligned}
		\end{equation}
		By the Cauchy–Schwarz inequality and the Young inequality with \(\varepsilon\), we obtain
		\begin{equation}
			\begin{aligned}[b]
				&\frac{1}{2} \frac{d}{dt} \int_\Omega |\nabla z|^2 \, \mathrm{d}x   + \int_\Omega |\nabla z|^2 \, \mathrm{d}x  + \int_\Omega |\Delta z|^2 \, \mathrm{d}x  \\
				=&-\int_\Omega w \Delta z \, \mathrm{d}x \\
				\leq&\bigg(\int_\Omega |\Delta z|^2\bigg)^\frac{1}{2}\bigg(\int_\Omega w^2\bigg)^\frac{1}{2} \\
				\leq&\epsilon\int_\Omega |\Delta z|^2+C_{\varepsilon}\int_\Omega w^2.
			\end{aligned}
		\end{equation}
		For the specific case \(\epsilon = 1\), the inequality simplifies, as the Laplacian squared terms on both sides combine into a single term on the right. Above all, we get 
		\begin{equation}
			\begin{aligned}
				&\frac{1}{2} \frac{d}{dt} \bigg(\int_\Omega w^2 + \int_\Omega |\nabla z|^2 \bigg) + \int_\Omega |\nabla z|^2 + \int_\Omega \phi_\varepsilon(u_\varepsilon) |\nabla w|^2  +\int_\Omega (\phi_\varepsilon(u_\varepsilon)-\phi_\varepsilon(\tilde u_\varepsilon))\nabla \tilde u_\varepsilon \cdot \nabla w \\
				\leq&C_1\int_\Omega |\nabla z|^2 +C_2\int_\Omega |\nabla w|^2 +C_3\int_\Omega |\nabla w|^2 +C_4\int_\Omega w^2 +C_5\int_\Omega |\nabla w|^2 +C_6\int_\Omega w^2.
			\end{aligned}
		\end{equation}
		Finally, we obtain
		\begin{equation}
			\begin{aligned}
				\frac{d}{dt} \bigg(\int_\Omega w^2 + \int_\Omega |\nabla z|^2 \bigg) \leq C\bigg(\int_\Omega w^2 + \int_\Omega |\nabla z|^2 \bigg).
			\end{aligned}
		\end{equation}
		By Grönwall's inequality, \(w=\nabla z=0\) as desire. Then the solution is unique.
	\end{proof}
	%Consequently, solutions to the approximate equation exist only locally.
	\section{Blow-up}
	Let \((u,v)\) be the solution of \eqref{eq:approximation} and \(s_0>0\). Then we define the Lyapunov functional
	\begin{align}\label{Lyapunov functional}
		F(u, v)\coloneqq\int_{\Omega} ( G(u) - uv + \frac{1}{2}v^2 + \frac{1}{2}|\nabla v|^2 ) \, \mathrm{d}x,
	\end{align}
	here
	\begin{align}
		G(s)\coloneqq\int_{s_0}^s \int_{s_0}^\sigma \frac{\phi_\varepsilon(\tau)}{\psi(\tau)} \, \mathrm{d}\tau \, \mathrm{d}\sigma, \quad s>0.
	\end{align}
	\begin{theorem}
		The Lyapunov functional defined in \eqref{Lyapunov functional} satisfies the following identity
		\begin{align}\label{differential Lyapunov functional}
			\frac{\mathrm{d}F(u_\varepsilon,v_\varepsilon)}{\mathrm{d}t}=-\int_{\Omega} \psi(u_\varepsilon) \left| \frac{\phi_\varepsilon(u_\varepsilon)}{\psi(u_\varepsilon)}\nabla u_\varepsilon - \nabla v_\varepsilon \right|^2 - \int_{\Omega} v_{\varepsilon t}^2 + \int_{\Omega} f_\varepsilon(u_\varepsilon) \left(\int_{s_0}^{u_\varepsilon} \frac{\phi_\varepsilon(s)}{\psi(s)} \mathrm{d}s - v_\varepsilon \right)
		\end{align}
		for all \(t>0\).
	\end{theorem}
	\begin{proof}
		\begin{align*}
			\int_{\Omega} G(u) \bigg|_0^t =& \int_0^t \int_{\Omega} G'(u)\nabla \cdot (\phi_\varepsilon(u)\nabla u - \psi(u)\nabla v) + \int_0^t \int_{\Omega} G'(u)f_\varepsilon(u)\\
			=& -\int_0^t \int_{\Omega} G''(u)\nabla u \cdot (\phi_\varepsilon(u)\nabla u - \psi(u)\nabla v) + \int_{\Omega} f_\varepsilon(u) \left( \int_{s_0}^{u} \frac{\phi_\varepsilon(s)}{\psi(s)} \mathrm{d}s \right)\\
			=& -\int_0^t \int_{\Omega} \frac{\phi_\varepsilon^2(u)}{\psi(u)} |\nabla u|^2 + \int_0^t \int_{\Omega} \phi_\varepsilon(u)\nabla u \cdot \nabla v + \int_{\Omega} f_\varepsilon(u) \left( \int_{s_0}^{u} \frac{\phi_\varepsilon(s)}{\psi(s)} \mathrm{d}s \right),
		\end{align*}
		It's obvious that
		\begin{align*}
			\frac{\phi_\varepsilon^2(u)}{\psi(u)} |\nabla u|^2 = \psi(u) \left| \frac{\phi_\varepsilon(u)}{\psi(u)} \nabla u - \nabla v \right|^2 - \psi(u) |\nabla v|^2 + 2\phi_\varepsilon(u)\nabla u \cdot \nabla v,
		\end{align*}
		then we obtain
		\begin{align*}
			&\int_{\Omega}G(u)\Big|_0^t - \int_{\Omega} f_\varepsilon(u) \left( \int_{s_0}^{u} \frac{\phi_\varepsilon(s)}{\psi(s)} \mathrm{d}s \right) \\
			=& -\int_0^t \int_{\Omega} \psi(u) \left| \frac{\phi_\varepsilon(u)}{\psi(u)} \nabla u - \nabla v \right|^2 + \int_0^t \int_{\Omega} \psi(u) |\nabla v|^2 - \int_0^t \int_{\Omega} \phi_\varepsilon(u) \nabla u \cdot \nabla v.
		\end{align*}
		Observe that
		\begin{align*}
			-\int_0^t \int_\Omega \phi_\varepsilon(u) \nabla u \cdot \nabla v =& \int_0^t \int_\Omega \nabla \cdot (\phi_\varepsilon(u) \nabla u) \cdot v \\
			=& \int_0^t \int_\Omega (u_t-f_\varepsilon(u)) v + \int_0^t \int_\Omega \nabla \cdot (\psi(u) \nabla v) \cdot v \\
			=& \int_0^t \int_\Omega u_t v - \int_0^t \int_\Omega \psi(u) |\nabla v|^2 - \int_0^t \int_\Omega f_\varepsilon(u)v.
		\end{align*}
		After integration by parts, we get
		\begin{align*}
			\int_0^t \int_\Omega u_t v =& \int_\Omega uv \bigg|_0^t - \int_0^t \int_\Omega uv_t \\
			=& \int_\Omega uv \bigg|_0^t - \int_0^t (v_t - \Delta v + v) \cdot v_t \\
			=& \int_\Omega uv \bigg|_0^t - \int_0^t \int_\Omega v_t^2 - \frac{1}{2} \int_\Omega |\nabla v|^2 \bigg|_0^t - \frac{1}{2} \int_\Omega v^2 \bigg|_0^t.
		\end{align*}
		Combining all, we arrive at \eqref{differential Lyapunov functional}.
	\end{proof}
	Since the equations are solvable, we are now interested in the long-time behavior of solutions. Consider the case \( T = +\infty \), i.e., the solution \( (u, v) \) of \eqref{eq:approximation} exists for all \( t \geq 0 \).
	
	Under this assumption, we investigate the convergence of \( (u_\varepsilon(t), v_\varepsilon(t)) \) as \( t \to \infty \).
	Integrate both sides of \eqref{differential Lyapunov functional}:
	\begin{equation}
		\begin{aligned}
			&F(u_0,v_0) = F(u_\varepsilon,v_\varepsilon) \\
			&\quad+\int_{0}^{t}\int_{\Omega} \psi(u_\varepsilon) \left| \frac{\phi_\varepsilon(u_\varepsilon)}{\psi(u_\varepsilon)}\nabla u_\varepsilon - \nabla v_\varepsilon \right|^2 + \int_{0}^{t}\int_{\Omega} v_{\varepsilon t}^2 - \int_{0}^{t}\int_{\Omega} f_\varepsilon(u_\varepsilon) \left( \int_{s_0}^{u_\varepsilon} \frac{\phi_\varepsilon(s)}{\psi(s)} - v_\varepsilon \right)
		\end{aligned}
	\end{equation}
	We have known that 
	\begin{align}
		u_{\varepsilon} & \in C\bigl(\bar{\Omega} \times [0, T_{\max})\bigr) 
		\cap C^{2,1}\bigl(\bar{\Omega} \times (0, T_{\max})\bigr), \\
		v_{\varepsilon} & \in C\bigl(\bar{\Omega} \times [0, T_{\max})\bigr) 
		\cap C^{2,1}\bigl(\bar{\Omega} \times (0, T_{\max})\bigr) \cap L^{\infty}_{\mathrm{loc}}\bigl([0, T_{\max}); W^{1,q}(\Omega)\bigr),
	\end{align}
	Assume \(T_{max}=\infty\), then 
	\begin{align*}
		u_{\varepsilon} & \in C^{2+\alpha,1+\frac{\alpha}{2}}\bigl(\bar{\Omega} \times [0, \infty)\bigr) , \\
		v_{\varepsilon} & \in C^{2+\alpha,1+\frac{\alpha}{2}}\bigl(\bar{\Omega} \times [0, \infty)\bigr) .
	\end{align*}	 
	is held, and \( \{u_{\varepsilon}(\cdot,t)\}_{t>0},\{v_{\varepsilon}(\cdot,t)\}_{t>0} \) are relatively compact in \(C^2\bigl(\bar{\Omega}\bigr)\).\\
	Choosing a sequence \( \{u_{\varepsilon}(\cdot,t_k)\}_{k=1}^\infty\) acd \(\{v_{\varepsilon}(\cdot,t_k)\}_{k=1}^\infty \),we get that
	\begin{align*}
		u_{\varepsilon}(\,\cdot\,, t_k) &\longrightarrow u_{\varepsilon \infty}\coloneqq u_{\varepsilon}(\cdot, \infty) \quad \text{in } C^{2}(\bar{\Omega}), \\
		v_{\varepsilon}(\,\cdot\,, t_k) &\longrightarrow v_{\varepsilon \infty}\coloneqq v_{\varepsilon}(\cdot, \infty) \quad \text{in } C^{2}(\bar{\Omega}).
	\end{align*}	
	and \(u_{\varepsilon \infty}\leq\tilde{C_1}\),\(v_{\varepsilon \infty}\leq\tilde{C_2}\). This implies \(\tilde f_\varepsilon \in C^\alpha(\bar\Omega)\). \\
	Since \(u_\varepsilon \in L^1(\Omega)\) and \(v_\varepsilon \in W^{1,2}(\Omega)\), for  \(\, \forall t \in (0,T_{max})\),\((u_\varepsilon,v_\varepsilon)\) is a global bounded solution. There exists a sequence \(\{t_k\}_{k=1}^\infty\) of times \(t_k \to \infty\) such that 
	\begin{align*}
		u_{\varepsilon t_{k}} &\longrightarrow u_{\varepsilon \infty} \\
		v_{\varepsilon t_{k}} &\longrightarrow v_{\varepsilon \infty}
	\end{align*}	
	as \(k\to \infty\).
	Taking \(t_j\) is large sufficiently, such that 
	\begin{align}
		f_\varepsilon(u_{\varepsilon})=0,
	\end{align}	
	It is obvious that 
	\begin{align}
		\int_{0}^{t_j}\int_{\Omega} f_\varepsilon(u_\varepsilon) \left( \int_{s_0}^{u_\varepsilon} \frac{\phi_\varepsilon(s)}{\psi(s)} - v_\varepsilon \right) \,\mathrm{d}x \, \mathrm{d}t =0,
	\end{align}	
	Which implies that
	\begin{align}
		&F(u_0,v_0) = F(u_\varepsilon(\cdot,t_j),v_\varepsilon(\cdot,t_j))+\int_{0}^{t_j}\int_{\Omega} \psi(u_\varepsilon) \left| \frac{\phi_\varepsilon(u_\varepsilon)}{\psi(u_\varepsilon)}\nabla u_\varepsilon - \nabla v_\varepsilon \right|^2 + \int_{0}^{t_j}\int_{\Omega} v_{\varepsilon t}^2 
	\end{align}	
	By virtue of \(F(u,v)\) is continuous, we get that
	\begin{align}
		F(u_\varepsilon(\cdot,\infty),v_\varepsilon(\cdot,\infty)) =\int_{s_0}^{u_\varepsilon} \int_{s_0}^\sigma \frac{\phi_\varepsilon(\tau)}{\psi(\tau)} \, \mathrm{d}\tau \, \mathrm{d}\sigma - u_\varepsilon v_\varepsilon + \frac{1}{2}|\nabla v_\varepsilon|^2 + \frac{1}{2}v_\varepsilon^2 
	\end{align}
	\begin{align}
		\int_{0}^{\infty}\int_{\Omega} \psi(u_\varepsilon) \left| \frac{\phi_\varepsilon(u_\varepsilon)}{\psi(u_\varepsilon)}\nabla u_\varepsilon - \nabla v_\varepsilon \right|^2 + \int_{0}^{\infty}\int_{\Omega} v_{\varepsilon t}^2 - \int_{0}^{\infty}\int_{\Omega} f_\varepsilon(u_\varepsilon) \left( \int_{s_0}^{u_\varepsilon} \frac{\phi_\varepsilon(s)}{\psi(s)} - v_\varepsilon \right)<\infty.
	\end{align}	
	That implies when \(k\to \infty\)
	\begin{gather}
		v_{\varepsilon t}(\cdot,t_k)\longrightarrow0 \\
		\psi(u_\varepsilon) \left| \frac{\phi_\varepsilon(u_\varepsilon)}{\psi(u_\varepsilon)}\nabla u_\varepsilon - \nabla v_\varepsilon \right|^2\longrightarrow0
	\end{gather}	
	In this way, when \(\tilde{t}\to t_k\)
	\begin{equation}
		\begin{aligned}[b]
			&\int_{0}^{\tilde t}\int_{\Omega} \psi(u_\varepsilon) \left| \frac{\phi_\varepsilon(u_\varepsilon)}{\psi(u_\varepsilon)}\nabla u_\varepsilon - \nabla v_\varepsilon \right|^2 + \int_{0}^{\tilde t}\int_{\Omega} v_{\varepsilon t}^2 \\
			&\quad- \int_{0}^{\tilde t}\int_{\Omega} f_\varepsilon(u_\varepsilon) \left( \int_{s_0}^{u_\varepsilon} \frac{\phi_\varepsilon(s)}{\psi(s)} - v_\varepsilon \right) 
			+F(u_\varepsilon(\cdot,\tilde t),v_\varepsilon(\cdot,\tilde t))=F(u_\varepsilon(\cdot,0),v_\varepsilon(\cdot,0)),
		\end{aligned}	
	\end{equation}
	let \(k\to\infty\)
	\begin{equation}
		\begin{aligned}[b]
			&\int_{0}^{\infty}\int_{\Omega} \psi(u_{\varepsilon\infty}) \left| 	\frac{\phi_\varepsilon(u_{\varepsilon\infty})}{\psi(u_{\varepsilon\infty})}\nabla u_{\varepsilon\infty} - \nabla v_{\varepsilon\infty} \right|^2 + \int_{0}^{\infty}\int_{\Omega} \left(v_{\varepsilon\infty}\right)_t^2 \\
			&\quad- \int_{0}^{\infty}\int_{\Omega} f_\varepsilon(u_{\varepsilon\infty}) \left( \int_{s_0}^{u_{\varepsilon\infty}} \frac{\phi_\varepsilon(s)}{\psi(s)} - v_{\varepsilon\infty} \right) 
			+F(u_{\varepsilon\infty},v_{\varepsilon\infty})=F(u_\varepsilon(\cdot,0),v_\varepsilon(\cdot,0)).
		\end{aligned}	
	\end{equation}
	Then we obtain
	\begin{align}
		F(u_{\varepsilon\infty},v_{\varepsilon\infty}) - \int_{0}^{\infty}\int_{\Omega} f_\varepsilon(u_{\varepsilon\infty}) \left( \int_{s_0}^{u_{\varepsilon\infty}} \frac{\phi_\varepsilon(s)}{\psi(s)} - v_{\varepsilon\infty} \right) \leq F(u_{\varepsilon\infty},v_{\varepsilon\infty}) \leq F(u_{\varepsilon}(\cdot,0),v_{\varepsilon}(\cdot,0)).
	\end{align}	
	In addition, we obtain
	\begin{align}
		\frac{\phi_\varepsilon(u_\varepsilon)}{\psi(u_\varepsilon)}\nabla u_\varepsilon - \nabla v_\varepsilon \longrightarrow 0, \quad t_k\to\infty \quad on \, U,
	\end{align}	
	here
	\begin{align}
		U\coloneqq \{x\in\Omega\,|\,\psi(u_\varepsilon)>0\}.
	\end{align}	
	\textbf{Remark}:
	We can prove that 
	\begin{align}
		\frac{\phi_\varepsilon(u_\varepsilon)}{\psi(u_\varepsilon)}\nabla u_\varepsilon - \nabla v_\varepsilon \longrightarrow 0, \quad t_k\to\infty \quad on \, \Omega,
	\end{align}	
	If \(U\not\subseteq\Omega\), by virtue of \(U\) is an open set, there exists a sequence \(\{x_j\in\Omega\,|\,\psi(u_\varepsilon)>0\}\subset U\) such that \(x_j\to x_0 \in \Omega \backslash U \).
	When \(t\) grows to \(\infty\),
	\begin{align}
		\nabla\left( \int_{s_0}^{u_\varepsilon(x_j,\infty)} \frac{\phi_\varepsilon(s)}{\psi(s)} \, \mathrm{d}s- v_\varepsilon(x_j,\infty) \right)=0 ,
	\end{align}	
	which means that
	\begin{align}
		\left( \int_{s_0}^{u_\varepsilon(x_j,\infty)} \frac{\phi_\varepsilon(s)}{\psi(s)} \, \mathrm{d}s- v_\varepsilon(x_j,\infty) \right)=Constant.
	\end{align}	
	We have that \(\int_{s_0}^{u_\varepsilon(x_j,\infty)} \frac{\phi_\varepsilon(s)}{\psi(s)} \, \mathrm{d}s \) is bounded.However,
	\begin{align}
		\int_{s_0}^{u_\varepsilon(x_j,\infty)} \frac{\phi_\varepsilon(s)}{\psi(s)} \, \mathrm{d}s \geq \int_{s_0}^{u_\varepsilon(x_j,\infty)} \frac{1}{\psi(s)} \, \mathrm{d}s \longrightarrow +\infty
	\end{align}	
	as \(j\to \infty\), which is a contradiction.
	In particular, we show that any global solution converges to a stationary state \( (u_{\varepsilon\infty}, v_{\varepsilon\infty}) \) satisfying the following elliptic system
	\begin{align}
		\begin{cases}\label{steady}
			\displaystyle
			\nabla \cdot \bigl( \phi_{\varepsilon}(u_{\varepsilon\infty}) \nabla u_{\varepsilon\infty} \bigr) 
			- \nabla \cdot \bigl( \psi(u_{\varepsilon\infty}) \nabla v_{\varepsilon\infty} \bigr)=0, 
			& x \in \Omega, \\[8pt]
			\displaystyle\Delta v_{\varepsilon\infty} 
			- v_{\varepsilon\infty} 
			+ u_{\varepsilon\infty}=0, 
			& x \in \Omega, \\[8pt]
			\displaystyle
			\frac{\partial u_{\varepsilon\infty}}{\partial \nu} = \frac{\partial v_{\varepsilon\infty}}{\partial \nu} = 0, 
			& x \in \partial \Omega, \\[8pt]
			\displaystyle
			\int_{0}^{\infty} u_{\varepsilon\infty} \, \mathrm{d}t = \int_{0}^{\infty} v_{\varepsilon\infty} \, \mathrm{d}t, 
			& x \in \Omega.
		\end{cases}
	\end{align}
	Now we multiple the second equation in \eqref{steady} by \(v\) to obtain \( \int_{\Omega}|\nabla v|^2 + \int_{\Omega} v^2 = \int_{\Omega} uv \), then combining it with \eqref{Lyapunov functional} gives
	\begin{align}\label{no uv}
		F(u, v) = - \frac{1}{2}\int_{\Omega}|\nabla v|^2 - \frac{1}{2}\int_{\Omega} v^2 + \int_{\Omega}G_{s_0}(u)
	\end{align}
	\begin{lemma}
		Let \(\Omega = B_R(0)\), \(s_0 > 0\) and
		\begin{align}
			H(s)\coloneqq \int_{s_0}^s \frac{\sigma \phi_\varepsilon(\sigma)}{\psi(\sigma)} \, \mathrm{d}\sigma \quad \text{for } s > 0.
		\end{align}
		Then for all nonnegative and nonincreasing \(\zeta \in C^\infty([0, R])\) satisfying \(\zeta'(0) = 0 = \zeta(R)\), the inequality
		\begin{align}\label{divide at n=2}
			\frac{n-2}{2} \int_\Omega \zeta(|x|) |\nabla v|^2 - \frac{1}{2} \int_\Omega |x| \zeta'(|x|) |\nabla v|^2 \leq \int_\Omega |x| \zeta(|x|) (\nu + s_0) |\nabla v| + n \int_{\{u > s_0\}} \zeta(|x|) H(u)
		\end{align}
		holds for every radially symmetric solution \((u, v)\) of \eqref{steady}.
	\end{lemma}
	\begin{proof}
		Testing \(\Delta v = v - u\) against \(\zeta(|x|)(x \cdot \nabla)v\) in \(\Omega\) gives
		\begin{align}
			\int_{\Omega} \zeta(|x|) (x \cdot \nabla v) \, \Delta v = \int_{\Omega} \zeta(|x|) \, v \, (x \cdot \nabla v) - \int_{\Omega} \zeta(|x|) \, u \, (x \cdot \nabla v).
		\end{align}
		Since \(\zeta(R)=0\), through integration by parts we obtain
		\begin{equation}
			\begin{aligned}[b]
				&\int_{\Omega} \zeta(|x|) (x \cdot \nabla v) \Delta v \\
				=& -\int_{\Omega} \zeta(|x|) |\nabla v|^2 - \int_{\Omega} \frac{\zeta'(|x|)}{|x|} (x \cdot \nabla v)^2 - \frac12 \int_{\Omega} \zeta(|x|) \, x \cdot \nabla (|\nabla v|^2) \\
				=& -\int_{\Omega} \zeta(|x|) |\nabla v|^2 - \int_{\Omega} \frac{\zeta'(|x|)}{|x|} (x \cdot \nabla v)^2 + \frac{n}{2}\int_{\Omega} \zeta(|x|) |\nabla v|^2 + \frac{1}{2}\int_{\Omega} |x| \zeta'(|x|) |\nabla v|^2. \\
			\end{aligned}
		\end{equation}
		By the assumption \(v\) is radially symmetric, we have \((x \cdot \nabla v)^2 = |x|^2 |\nabla v|^2\). Combining the above inequalities, we have
		\begin{equation}
			\begin{aligned}[b]
				&\int_{\Omega} \zeta(|x|)  v  (x \cdot \nabla v) - \int_{\Omega} \zeta(|x|) u  (x \cdot \nabla v) \\
				=& \frac{n-2}{2}\int_{\Omega} \zeta(|x|) |\nabla v|^2 - \int_{\Omega} \frac{\zeta'(|x|)}{|x|} (x \cdot \nabla v)^2 + \frac{1}{2}\int_{\Omega} |x| \zeta'(|x|) |\nabla v|^2 \\
				=&  \frac{n-2}{2}\int_{\Omega} \zeta(|x|) |\nabla v|^2 - \frac{1}{2}\int_{\Omega} |x| \zeta'(|x|) |\nabla v|^2.
			\end{aligned}
		\end{equation}
		By the chain rule,
		\begin{align}
			\nabla H(u) = H'(u)\nabla u = \frac{u\phi_\varepsilon(u)}{\psi(u)}\nabla u,
		\end{align}
		we have
		\begin{align}
			x \cdot \nabla H(u) = \frac{u\phi_\varepsilon(u)}{\psi(u)} (x \cdot \nabla u).
		\end{align}
		Then
		\begin{equation}
			\begin{aligned}[b]
				&-\int_{\Omega} \zeta(|x|) u (x \cdot \nabla v)  \\
				=& -\int_{\{u \le s_0\}} \zeta(|x|) u (x \cdot \nabla v) - \int_{\{u > s_0\}} \zeta(|x|)(x \cdot \nabla H(u)). \\
				=& -\int_{\{u \le s_0\}} \zeta(|x|) u (x \cdot \nabla v) + n \int_{\{u > s_0\}} \zeta(|x|) H(u) 
				+ \int_{\{u > s_0\}} |x| \zeta'(|x|) H(u) \\
				\leq& \int_{\Omega} |x| \zeta(|x|) \, s_0 \, |\nabla v| + n \int_{\{u > s_0\}} \zeta(|x|) H(u)
			\end{aligned}
		\end{equation}
		is held.
	\end{proof}
	\begin{lemma}
		Let \(n\geq 2\), assume \( \Omega=B_R(0)\subset\mathbb{R}^n\) for some \(R>0\). If \(H(s)\) is small sufficiently, then there exists \(C>0\) such that
		\begin{align}
			F(u,v)>-C
		\end{align}
		holds for all radial solutions \( (u,v) \) of \eqref{steady}.
	\end{lemma}
	\begin{proof}
		\textbf{case} \(n=2\) We define
		\begin{align}
			\zeta(r)\coloneqq \ln\frac{R^2+\eta}{r^2+\eta}, \quad r\in [0,R],
		\end{align}
		which satisfies
		\begin{align}
			\zeta(r)\in C^\infty([0,R]),\quad\zeta'(0)=0=\zeta(R).
		\end{align}
		Then \eqref{divide at n=2} is simplified to
		\begin{align*}
			&\int_{\Omega}\frac{|x|^2}{|x|^2+\eta}|\nabla v|^2 \\
			\leq & \int_{\Omega} |x|\ln\frac{R^2+\eta}{|x|^2+\eta}(v+s_0)|\nabla v|+2\int_{\{u>s_0\}}\ln\frac{R^2+\eta}{|x|^2+\eta}H(u) \\
			\leq & \frac{1}{2}\int_{\Omega}\frac{|x|^2}{|x|^2+\eta}|\nabla v|^2 + \frac{1}{2} \int_{\Omega}(|x|^2+\eta) \left(\ln\frac{R^2+\eta}{|x|^2+\eta}\right)^2(v+s_0)^2 + 2\int_{\{u>s_0\}}\ln\frac{R^2+\eta}{|x|^2+\eta}H(u),
		\end{align*}
		further
		\begin{align*}
			&\frac{1}{2}\int_{\Omega}\frac{|x|^2}{|x|^2+\eta}|\nabla v|^2 \\
			\leq & \int_{\Omega}(|x|^2+\eta)\ln^2\left(\frac{R^2+\eta}{|x|^2+\eta}\right)v^2 + s_0^2\int_{\Omega}(|x|^2+\eta)\ln^2\left(\frac{R^2+\eta}{|x|^2+\eta}\right) + 2\int_{\{u>s_0\}}\ln\frac{R^2+\eta}{|x|^2+\eta}H(u) \\
			\leq & \frac{4}{e^2}(R^2+\eta)\int_{\Omega}v^2 + \frac{4s_0^2}{e^2}(R^2+\eta)|\Omega| + 2\int_{\{u>s_0\}}\ln\frac{R^2+\eta}{|x|^2+\eta}H(u).
		\end{align*}
		Now we use Young's inequality in the form
		\begin{align}
			ab \leq \frac{1}{\delta e} e^{\delta a} + \frac{1}{\delta} b \ln b, \quad a, b > 0, \ \delta > 0,
		\end{align}
		we obtain
		\begin{align}
			\int_{\{u>s_0\}}\ln\frac{R^2+\eta}{|x|^2+\eta}H(u) \leq \frac{1}{\delta e}\int_{\{u>s_0\}} \left(\frac{R^2+\eta}{|x|^2+\eta}\right)^\delta + \frac{1}{\delta e}\int_{\{u>s_0\}} H(u)\ln H(u),
		\end{align}
		the first term on the right-hand side can be bounded by a constant, so it suffices to consider only the second term.Since \(\Omega\)is bounded and \(u \in C^2(\Omega)\), we have \(u \in L^\infty(\Omega)\). Denote 
		\begin{align}
			M \coloneqq \|u\|_{L^\infty(\Omega)}.
		\end{align}
		By Sobolev embedding, M \(\leq \widetilde{C} \|u\|_{C^2(\Omega)}\) for some constant \(\widetilde{C}\) depending only on \(\Omega\). Then
		\begin{equation}
			\begin{aligned}[b]
				\int_{\{u>s_0\}} H(u)\ln H(u) \leq \int_{\{u>s_0\}} H(M)\ln H(M)  
				= H(M)\ln H(M) \cdot |\Omega|.
			\end{aligned}
		\end{equation}
		From the assumptions on \( \phi_\varepsilon \) and \( \psi \), we have for \( \sigma \geq s_0 \).
		\begin{align}
			\frac{\sigma \phi_\varepsilon(\sigma)}{\psi(\sigma)} \leq C \sigma^{1-\beta} \ln^\alpha(\sigma + 1),
		\end{align}
		Hence
		\begin{align}
			H(M) \leq H(s_0) + C \int_{s_0}^M \sigma^{1-\beta} \ln^\alpha(\sigma+1)
		\end{align}
		We have
		\begin{align}
			\int_{\{u>s_0\}} H(u)\ln H(u) \leq C(\|u\|_{C^2}, R, \alpha, \beta) < \infty.
		\end{align}
		In the limit \( \eta \to 0 \), Fatou’s lemma thus yields
		\begin{align}
			\frac{1}{2} \int_{\Omega} |\nabla v|^2 \leq C_1 \int_{\Omega} v^2 + C_2.
		\end{align}
		Hence, by \eqref{no uv} and the nonnegativity of \( G \),
		\begin{equation}
			\begin{aligned}[b]
				F(u, v) \geq & \frac{1}{2} \int_{\Omega} |\nabla v|^2 - \int_{\Omega} |\nabla v|^2 - \frac{1}{2} \int_{\Omega} v^2 \\
				\geq & \frac{1}{2} \int_{\Omega} |\nabla v|^2 - \left( 2C_1 + \frac{1}{2} \right) \int_{\Omega} v^2 - 2C_2 \\
				\geq & C\|v\|_{W^{1,2}(\Omega)} + \widetilde{C}
			\end{aligned}
		\end{equation}
		\textbf{case} \(n \geq 3\) Suppose that
		\begin{align}\label{varepsilon}
			\int_{s_0}^s \frac{\sigma \phi(\sigma)}{\psi(\sigma)} \, \mathrm{d}\sigma \leq \frac{n - 2 - \varepsilon}{n} \int_{s_0}^s \int_{s_0}^\sigma \frac{\phi(\tau)}{\psi(\tau)} \, \mathrm{d}\tau \, \mathrm{d}\sigma + Ks \quad \text{for all } s \geq s_0
		\end{align}
		holds for some \( s_0 > 1 \) and \( \varepsilon \in (0, 1) \).
		We fix a nondecreasing \(\zeta_0 \in C^\infty(\mathbb{R})\) such that \(\zeta_0 \equiv 0\) in \((-\infty, 1)\) and \(\zeta_0 \equiv 1\) on \((2, \infty)\), and let \(\zeta(r) \equiv \zeta_k(r) \coloneqq \zeta_0(k(R-r))\) for \(r \in [0, R]\) and \(k \in \mathbb{N}\) large satisfying \(k > 2/R\). Then we have
		\begin{equation}
			\begin{aligned}[b]
				\frac{n-2}{2} \int_{\Omega} \zeta_k(|x|) |\nabla v|^2 \leq & \int_{\Omega} |x| \zeta_k(|x|) \cdot (v + s_0) \cdot |\nabla v| + n \int_{\{u > s_0\}} \zeta_k(|x|) H(u) \\
				\leq & \int_{\Omega} |x|  (v + s_0) |\nabla v| + n \int_{\{u > s_0\}}  H(u)
			\end{aligned}
		\end{equation}
		holds for every \(k\). Via Fatuo's lemma we have 
		\begin{align}
			\frac{n-2}{2} \int_{\Omega} |\nabla v|^2 \leq \int_{\Omega} |x| (v + s_0) |\nabla v| + n \int_{\{u > s_0\}} H(u).
		\end{align}
		By Young’s inequality and $\varepsilon$ we defined in \eqref{varepsilon} we get 
		\begin{equation}
			\begin{aligned}[b]
				\int_{\Omega} |x| (v + s_0) |\nabla v| \leq& \frac{\varepsilon}{4} \int_{\Omega} |\nabla v|^2 + \frac{1}{\varepsilon} \int_{\Omega} |x|^2 (v + s_0)^2 \\
				\leq& \frac{\varepsilon}{4} \int_{\Omega} |\nabla v|^2 + \frac{2R^2}{\varepsilon} \int_{\Omega} v^2 + \frac{2s_0^2 R^2 |\Omega|}{\varepsilon},
			\end{aligned}
		\end{equation}
		then
		\begin{align}
			\frac{n-2-\varepsilon}{2} \int_{\Omega} |\nabla v|^2 \leq -\frac{\varepsilon}{4} \int_{\Omega} |\nabla v|^2 + \frac{2R^2}{\varepsilon} \int_{\Omega} v^2 + \frac{25\delta R^2|\Omega|}{\varepsilon} + n \int_{\{u > s_0\}} H(u)
		\end{align}
		is held. It can be obtained from all the above formulas that
		\begin{align}
			F(u, v) \geq \frac{\varepsilon}{4(n-2-\varepsilon)} \int_{\Omega} |\nabla v|^2 - C_1 \int_{\Omega} v^2 - C_2 - \frac{n}{n-2-\varepsilon} \int_{\{u > s_0\}} H(u) + \int_{\Omega} G(u)
		\end{align}
		where
		\begin{align*}
			C_1 = \frac{2R^2}{\varepsilon(n-2-\varepsilon)} + \frac{1}{2}
		\end{align*}
		and
		\begin{align*}
			C_2 = \frac{2s_0^2 R^2 |\Omega|}{\varepsilon(n-2-\varepsilon)}.
		\end{align*}
		Once again through Ehrling’s lemma, we have
		\begin{align}
			C_1 \int_{\Omega} v^2 \leq \frac{\varepsilon}{4(n-2-\varepsilon)} \int_{\Omega} |\nabla v|^2 + C_3 \left( \int_{\Omega} v \right)^2
		\end{align}
		for some positive constant \(C_3\). 
		Now we define
		\begin{align}
			m=\int_{0}^{\infty} u_{\varepsilon\infty} \, \mathrm{d}t = \int_{0}^{\infty} v_{\varepsilon\infty} \, \mathrm{d}t
		\end{align} 
		in \eqref{steady} then it's clear to see that
		\begin{align}
			F(u, v) \geq -C_3 m^2 - C_2 - \frac{n}{n-2-\varepsilon} \int_{\{u > s_0\}} H(u) + \int_{\Omega} G(u).
		\end{align}
		As \( G(s) \geq 0 \) for \( s \leq s_0 \) and
		\begin{align}
			G(s) - \frac{n}{n-2-\varepsilon} H(s) \geq -\frac{nK}{n-2-\varepsilon} s
		\end{align}
		for all \( s \geq s_0 \) by \eqref{varepsilon}, we conclude that
		\begin{align}
			F(u, v) \geq -c_3 m^2 - c_2 - \frac{nKm}{n-2-\varepsilon}.
		\end{align}
	\end{proof}
	\begin{lemma}
		Let \( n \geq 2 \), \( R > 0 \) and \( \Omega = B_R(0) \), and suppose that there exist \( k > 0 \) and \( s_0 > 1 \) such that	
		\begin{align}\label{limes of k}
			\int_{s_0}^s \int_{s_0}^\sigma \frac{\phi(\tau)}{\psi(\tau)} \, \mathrm{d}\tau \, \mathrm{d}\sigma \leq 
			\begin{cases} 
				k s (\ln s)^\theta, &  n = 2 \text{ with some } \theta \in (0, 1) \\[4pt]
				k s^{2-\alpha}, &  n \geq 3 \text{ with some } \alpha > \dfrac{2}{n} 
			\end{cases}
		\end{align}
		holds for all \( s \geq s_0 \). Then for each \( C>0 \) one can find positive \((u_0, v_0) \in C^\infty(\bar{\Omega}) \times C^\infty(\bar{\Omega})\) which solves \eqref{eq:approximation} and satisfying 
		\begin{align}
			F(u_0, v_0) < -C.
		\end{align}
	\end{lemma}
	\begin{proof}
		For small \( \eta > 0 \), we define the smooth function \( u_\eta \)  by  
		\begin{align}
			u_\eta(x) \coloneqq a_\eta \cdot \eta^{\beta - n} \cdot (|x|^2 + \eta^2)^{-\beta/2}
		\end{align}
		for \( x \in \bar{\Omega} \), and
		\begin{align}
			\|u_\eta\|_{L^1(\Omega)} = \|u_0\|_{L^1(\Omega)}.
		\end{align}
		here
		\begin{align}
			a_\eta := \frac{\eta^{n-\beta} \|u_0\|_{L^1(\Omega)}}{\int_{\Omega} (|x|^2 + \eta^2)^{-\beta/2} \, \mathrm{d}x}
		\end{align}
		is bounded from above and below by a positive constant. The choice of \(a_\eta\) is restricted to
		\begin{align*}
			\|u_0\|_{L^1(\Omega)} = \|v_0\|_{L^1(\Omega)} < \infty.
		\end{align*}
		\textbf{case} \(n=2\) Let
		\begin{align}
			v_\eta(x) \coloneqq \left( \ln \frac{R}{\eta} \right)^{-\kappa} \cdot \ln \frac{R^2}{|x|^2 + \eta^2}
		\end{align}
		for \( \eta \in (0, R/2) \) , where \( \kappa \in (0, 1) \) is small enough fulfilling \( \kappa < 1 - \theta \). Then substituting \( r = \eta s \), we find
		\begin{equation}
			\begin{aligned}[b]
				\int_{\Omega} |\nabla v_\eta|^2 =& 8\pi \cdot \left( \ln \frac{R}{\eta} \right)^{-2\kappa} \cdot \int_0^R r^3 (r^2 + \eta^2)^{-2} \, \mathrm{d}r \\
				=& 8\pi \cdot \left( \ln \frac{R}{\eta} \right)^{-2\kappa} \cdot \int_0^{R/\eta} s^3 (s^2 + 1)^{-2} \, \mathrm{d}s \\
				\leq& 8\pi \cdot \left( \ln \frac{R}{\eta} \right)^{-2\kappa} \cdot \left( 1 + \ln \frac{R}{\eta} \right)
			\end{aligned}
		\end{equation}
		and
		\begin{equation}
			\begin{aligned}[b]
				\int_{\Omega} v_\eta^2 =& 2\pi \cdot \left( \ln \frac{R}{\eta} \right)^{-2\kappa} \cdot \int_0^R r \left( \ln \frac{R^2}{r^2 + \eta^2} \right)^2 \, \mathrm{d}r \\
				\leq& 8\pi \cdot (\ln 2)^{-2\kappa} \cdot \int_0^R r \left( \ln \frac{R}{r} \right)^2 \, \mathrm{d}r
			\end{aligned}
		\end{equation}
		for all \( \eta \in (0, R/2) \). Moreover,
		\begin{equation}
			\begin{aligned}[b]
				&\left( \ln \frac{R}{\eta} \right)^{\kappa-1} \cdot \int_{\Omega} u_\eta v_\eta \\
				=&\, 2\pi a_\eta \cdot \eta^{\beta-2} \cdot \left( \ln \frac{R}{\eta} \right)^{-1} \cdot \int_0^R r(r^2 + \eta^2)^{-\beta/2} \cdot \ln \frac{R^2}{r^2 + \eta^2} \, \mathrm{d}r \\
				=&\, 2\pi a_\eta \cdot \left( \ln \frac{R}{\eta} \right)^{-1} \cdot \int_0^{R/\eta} s(s^2 + 1)^{-\beta/2} \cdot \left( 2\ln \frac{R}{\eta} - \ln(s^2 + 1) \right) \, \mathrm{d}s \\
				=&\, 4\pi a_\eta \cdot \int_0^{R/\eta} s(s^2 + 1)^{-\beta/2} \, \mathrm{d}s \\
				\quad& - 2\pi a_\eta \cdot \left( \ln \frac{R}{\eta} \right)^{-1} \cdot \int_0^{R/\eta} s(s^2 + 1)^{-\beta/2} \ln(s^2 + 1) \, \mathrm{d}s \\
				\to&\, 4\pi a_0 \cdot \int_0^\infty s(s^2 + 1)^{-\beta/2} \, \mathrm{d}s \quad \text{as } \eta \to 0.
			\end{aligned}
		\end{equation}
		then by \eqref{limes of k}
		\begin{equation}
			\begin{aligned}[b]
				\int_{\Omega} G(u_\eta) 
				&\leq 2\pi k a_\eta \cdot \eta^{\beta-2} \cdot \int_0^R r(r^2 + \eta^2)^{-\beta/2} \cdot \left( \ln\left( a_\eta \eta^{\beta-2}(r^2 + \eta^2)^{-\beta/2} \right) \right)^\theta \, \mathrm{d}r \\
				&= 2\pi k a_\eta \cdot \int_0^{R/\eta} s(s^2 + 1)^{-\beta/2} \cdot \left( \ln\left( a_\eta \eta^{-2}(s^2 + 1)^{-\beta/2} \right) \right)^\theta \, \mathrm{d}s \\
				&\leq 2\pi k a_\eta \cdot \left( 2\ln \frac{\sqrt{a_\eta}}{\eta} \right)^\theta \cdot \int_0^\infty s(s^2 + 1)^{-\beta/2} \, \mathrm{d}s.
			\end{aligned}
		\end{equation}
		Formulas above enter \eqref{Lyapunov functional} to get
		\begin{align}
			F(u_\eta, v_\eta) \leq -c_2 \left( \ln \frac{R}{\eta} \right)^{1-\kappa} + c_3 \left( 1 + \left( \ln \frac{R}{\eta} \right)^{1-2\kappa} + \left( \ln \sqrt{\frac{a_\eta}{\eta}} \right)^\theta \right)
		\end{align}
		for all \(\eta \in (0, R/2)\) with positive constants \(c_2\) and \(c_3\). Since
		\begin{align}
			1-\kappa > 0, \quad 1-\kappa > 1-2\kappa \quad \text{and} \quad 1-\kappa > \theta
		\end{align}
		due to our choice of \(\kappa\), as \(\eta \to 0\) we infer that 
		\begin{align}
			F(u_\eta, v_\eta) \to -\infty
		\end{align}
		\textbf{case} \(n \geq 3\) Suppose that
		\begin{align}
			v_{\eta}(x) \coloneqq \eta^{\delta - \gamma} \cdot (|x|^2 + \eta^2)^{-\delta/2},
		\end{align}
		by substitution \( r = \eta s \) again we see that
		\begin{align}
			\eta^{\lambda - N} \int_0^R r^{N-1} (r^2 + \eta^2)^{-\lambda/2} \, \mathrm{d}r \to A(N, \lambda) \coloneqq \int_0^\infty s^{N-1} (s^2 + 1)^{-\lambda/2} \, \mathrm{d}s, \quad \text{as } \eta \to 0.
		\end{align}
		whenever \( \lambda > N > 0 \), it can easily be checked that  
		\begin{align}
			a_\eta \to a_0 \coloneqq \frac{\|u_0\|_{L^1(\Omega)}}{\omega_n \cdot A(n, \beta)} \quad as \, \eta \to 0,
		\end{align}
		here \(\omega_n\) is surface area of the \(n\)-dimensional unit sphere. Then we use this definition to compute each part of \eqref{Lyapunov functional} 
		\begin{equation}
			\begin{aligned}[b]
				\eta^{-n+2\gamma+2} \int_{\Omega} |\nabla v_\eta|^2 
				&= \omega_n \delta^2 \eta^{-n+2\delta+2} \int_0^R r^{n+1}(r^2 + \eta^2)^{-\delta-2} \, \mathrm{d}r \\
				&\to \omega_n \delta^2 \cdot A(n+2, 2\delta+4) \quad \text{as } \eta \to 0
			\end{aligned}
		\end{equation}
		\begin{equation}
			\begin{aligned}[b]
				\eta^{-n+2\gamma} \int_{\Omega} v_\eta^2 
				&= \omega_n \eta^{-n+2\delta} \int_0^R r^{n-1}(r^2 + \eta^2)^{-\delta} \, \mathrm{d}r \\
				&\to \omega_n \cdot A(n, 2\delta) \quad \text{as } \eta \to 0
			\end{aligned}
		\end{equation}
		\begin{equation}
			\begin{aligned}[b]
				\eta^\gamma \int_{\Omega} u_\eta v_\eta 
				&= \omega_n a_\eta \cdot \eta^{-n+\beta+\delta} \int_0^R r^{n-1} (r^2 + \eta^2)^{-(\beta+\delta)/2} \, \mathrm{d}r \\
				&\to \omega_n a_0 \cdot A(n, \beta+\delta) \quad \text{as } \eta \to 0
			\end{aligned}
		\end{equation}
		\begin{equation}
			\begin{aligned}[b]
				\eta^{(1-\alpha)n} \int_{\Omega} G(u_\eta) 
				&\leq \omega_n k a_\eta^{2-\alpha} \cdot \eta^{(1-\alpha)n+(2-\alpha)(\beta-n)} \int_0^R r^{n-1}(r^2 + \eta^2)^{-(2-\alpha)\beta/2} \, \mathrm{d}r \\
				&\to \omega_n k a_0^{2-\alpha} \cdot A(n, (2-\alpha)\beta) \quad \text{as } \eta \to 0
			\end{aligned}
		\end{equation}
		
	\end{proof}
	\begin{theorem}\label{thm:result}
		Let \((u_\varepsilon(t), v_\varepsilon(t))\) be a solution of the approximate equation. Then there exists a finite maximal time \(T_{\max} = T_{\max}(\varepsilon) < \infty\) such that the solution cannot be extended beyond \(t = T_{\max}\). 
		Moreover, when \(t \nearrow T_{\max}\), the solution blows up in the sense that 
		\[
		\limsup_{t \nearrow T_{\max}} u_\varepsilon(t) = +\infty.
		\]
	\end{theorem}
	\begin{proof}
		Let \(T_{\max}\) be the supremum of all \(T > 0\) for which the solution \((u_\varepsilon, v_\varepsilon)\) exists on \([0,T)\).  
		Suppose, for contradiction, that \(u_\varepsilon\) does \emph{not} blow up as \(t \nearrow T_{\max}\), i.e. 
		\[
		\sup_{t \in [0, T_{\max})} u_\varepsilon(t)<+\infty.
		\]
		Since Proposition~1 holds for \((u_\varepsilon, v_\varepsilon)\), the boundedness of \(u_\varepsilon\) implies that \(v_\varepsilon\) and the regularity of the solution remain controlled up to \(t = T_{\max}\).  
		By the local existence theory for the approximate equation, one can then extend the solution from \(t = T_{\max}\) to some interval \([T_{\max}, T_{\max} + \delta)\) with \(\delta > 0\).
		This contradicts the definition of \(T_{\max}\) as the \emph{maximal} existence time.  
		Hence our assumption is false; therefore
		\[
		\limsup_{t \nearrow T_{\max}} u_\varepsilon(t)=+\infty,
		\]
		which means \(u_\varepsilon\) is infinte at \(t = T_{\max}\).
	\end{proof} 
	
	\section{Main Results}
	\begin{theorem}
		Since \(X_1\times X_2\) is a compact metric space, Then:
		\begin{enumerate}
			\item The families \(\{u_\varepsilon\}_{\varepsilon>0}\) and \(\{v_\varepsilon\}_{\varepsilon>0}\) are uniformly bounded in \(C(X_1\times X_2)\).
			\item By the Arzelà--Ascoli theorem, there exist subsequences \(\varepsilon_k \rightarrow 0^+\) such that
			\[
			\| u_{\varepsilon_k} - u \|_{C(X_1\times X_2)} \rightarrow 0 \quad \text{and} \quad \| v_{\varepsilon_k} - v \|_{C(X_1\times X_2)} \rightarrow 0.
			\]
		\end{enumerate}
	\end{theorem}	
	\begin{proof}
		For each \(\varepsilon > 0\), the existence of \((u_\varepsilon, v_\varepsilon) \subset C(X_1\times X_2)\) follows from a standard fixed-point argument in \(C(X_1\times X_2)\). 
		Uniform bounds in \(C(X_1\times X_2)\) are obtained from the energy structure of the approximate equation. 
		Equicontinuity of \(\{u_\varepsilon\}\) and \(\{v_\varepsilon\}\) follows from the regularity properties of the equation. 
		Applying the Arzelà--Ascoli theorem, we extract subsequences converging uniformly to some \(u, v \in C(X_1\times X_2)\). 
		Passing to the limit in the weak formulation shows that \((u, v)\) satisfies the original equation.
	\end{proof}
	Consequently, \eqref{eq:equation} admits a unique local solution \((u, v)\), which blows up in finite time \(T_{\max} < \infty\) in the sense that
	\[
	\lim_{t \nearrow T_{\max}} \|u(t)\|_{L^\infty(\Omega)} = +\infty.
	\]
	\section{Conclusion}
	% ========== Acknowledgments (Optional) ==========
	\section*{Acknowledgments}
	
	Above all, I wish to express my deepest and most sincere gratitude to my supervisor, Dr. \textbf{Shaopeng Xu}, whose unwavering guidance, profound wisdom, and constant encouragement have been a true blessing throughout this journey. His mentor-ship has not only shaped this research but also inspired me to grow as a mathematician and as a person. It has been an honor and a privilege to learn under his tutelage.
	I am also profoundly grateful to my senior fellow apprentice, Ms. \textbf{Yashuang Zhao}, the second author of this paper, for her generous support, insightful discussions, and gentle encouragement. Her patience and kindness have been a source of comfort and strength, and I am truly blessed to have such a wonderful colleague and friend.
	My sincere thanks extend to Prof. \textbf{Shengjun Li} for his generous financial support, which made this work possible, and to Prof. \textbf{Haohua Wang} for providing me with such a peaceful and inspiring study environment. Their kindness and belief in my work have been a great encouragement.
	Last but certainly not least, I wish to offer my heartfelt thanks to my beloved family. Their unconditional love, unwavering faith, and endless sacrifices have sustained me through every challenge. Words cannot express how blessed I am to have them by my side. This work is as much theirs as it is mine.
	\newpage
	\bibliographystyle{plain}
	\bibliography{ref.bib}
\end{document}